\documentclass[11pt,letterpaper]{amsart}

\usepackage[english]{babel}
\usepackage[OT2,T1]{fontenc} 

\usepackage{mathtools}
\usepackage{amssymb}
\usepackage{stmaryrd}
\usepackage{bm}
\usepackage{tikz-cd}
\usepackage[version=4]{mhchem}

\usepackage{comment}
\usepackage{xcolor}
\usepackage{enumitem}
\usepackage{graphicx}
\usepackage{setspace}
\usepackage{url}

\usepackage[title]{appendix}
\usepackage{chngcntr}
\usepackage{apptools}
\usepackage{thmtools}

\usepackage[pagebackref]{hyperref}
\usepackage[nameinlink]{cleveref}
\usepackage{bookmark}

\definecolor{crimsonglory}{rgb}{0.75, 0.0, 0.2}
\definecolor{darkpowderblue}{rgb}{0.0, 0.2, 0.6}

\hypersetup{
	colorlinks = true,
	urlcolor   = darkpowderblue,
	linkcolor  = darkpowderblue,
	citecolor  = crimsonglory
}

\theoremstyle{plain}
\newtheorem{theorem}{Theorem}[section]
\newtheorem{definition}[theorem]{Definition}

\newtheorem{prop}[theorem]{Proposition}

\newtheorem{claim}[theorem]{Claim}
\newtheorem{remark}[theorem]{Remark}
\newtheorem{remarks}[theorem]{Remarks}
\newtheorem{conj}[theorem]{Conjecture}

\DeclareMathOperator{\SL}{SL}

\DeclareMathOperator{\atyp}{Atyp}

\DeclareMathOperator{\ord}{ord}

\DeclareMathOperator{\ssing}{ssing}

\DeclareMathOperator{\bad}{bad}
\DeclareMathOperator{\cd}{cd}
\DeclareMathOperator{\strspec}{strspec}

\newcommand{\Sum}[2]{\displaystyle\sum_{#1}^{#2}}

\newcommand{\N}{\mathbb{N}}
\renewcommand{\C}{\mathbb{C}}
\newcommand{\Q}{\mathbb{Q}}

\newcommand{\CE}{\mathcal{E}}

\newcommand{\ld}{,\ldots,}

\DeclareSymbolFont{cyrletters}{OT2}{wncyr}{m}{n}
\DeclareMathSymbol{\Sha}{\mathalpha}{cyrletters}{"58}

\newcommand{\mathsym}[1]{{}}
\newcommand{\unicode}[1]{{}}

\title{Supersingular reduction and strongly special intersections in powers of the modular curve}
\author{Georgios Papas}
\address{Institute for Advanced Study\\
	1 Einstein Drive\\
	Princeton, N.J. 08540\\
	U.S.A.}
\email{gpapas@ias.edu}

\begin{document}
\subjclass[2020]{Primary 11G18; Secondary 14G35, 11G50}

\keywords{Zilber--Pink conjecture, modular curves, unlikely intersections, G-functions, Lang--Trotter conjecture}
	\begin{abstract} 
		We show that Lang--Trotter-type sparsity for simultaneous supersingular reduction of pairs of elliptic curves provides a new arithmetic input for unlikely intersections in powers of the modular curve. Assuming such a sparsity statement, we prove two Zilber--Pink-type finiteness results for Hodge generic curves in $Y(1)^n$. The proof proceeds through height bounds obtained by applying the $G$-function method of Yves Andr\'e. 
		\end{abstract}
\maketitle
				
\section{Introduction}\label{section:intro}

Following work of P. Habegger and J. Pila, finiteness questions of Zilber--Pink type for curves in $Y(1)^n$ may be reduced to suitable height bounds. In the $G$-functions method of Yves Andr\'e, such height bounds are obtained by constructing global algebraic relations among values of $G$-functions. In this approach, a persistent difficulty is bounding the contribution of places at which the relevant elliptic curves have simultaneous supersingular reduction. The purpose of this paper is to isolate an arithmetic input that controls this contribution.

We show that a Lang--Trotter-type sparsity statement for primes of simultaneous supersingular reduction provides such an input. Under this conjectural assumption, we prove two finiteness results for intersections of curves with natural, explicit families of strongly special subvarieties in $Y(1)^n$. Here, following Habegger--Pila \cite{habeggerpila1}, a special subvariety is called strongly special if its defining equations involve no condition of the form $x_i=j_0$, where $x_1,\ldots,x_n$ denote the coordinate functions on $Y(1)^n$ and $j_0$ is a singular modulus.

We introduce the following notation for the special subvarieties considered in the main theorems. 

\begin{definition}\label{defnspecialsubs} 
For $n\geq 3$, let $I:=\{i_1,i_2,i_3\}\subset \{1\ld n\}$ be a subset with $\# I=3$, and let $N,N'\in\N$. 

$1$. We write $V(I;N,N')$ for the subvariety of $Y(1)^n$ defined by the equations $\Phi_N(x_{i_1},x_{i_2})=0$ and $\Phi_{N'}(x_{i_1},x_{i_3})=0$, where $\Phi_N(X,Y)$ denotes the $N$-th modular polynomial. 

$2$. Given another such subset $J\subset \{1\ld n\}$ with $I\cap J=\emptyset$ and $M,M'\in \N$, we define 
\begin{equation} 
	V(I,J;N,N',M,M'):=V(I;N,N')\cap V(J;M,M'). 
\end{equation} 

$3$. Given $j\in \{1\ld n\}\backslash I$ and $N''\in \N$, we define \begin{equation} V(I,j;N,N',N''):=V(I;N,N')\cap V(\Phi_{N''}(x_{i_1},x_j)). 
\end{equation}
\end{definition}

With this notation, the first finiteness result we establish is the following. 

\begin{theorem}\label{thmzpiny(1)n} 
	Let $S\subset Y(1)^n$, with $n\geq 4$, be a smooth irreducible curve defined over $\bar{\Q}$ that is not contained in a proper special subvariety of $Y(1)^n$. Let $I$ be a subset of $\{1\ld n\}$ with $\# I=3$. 
	
Assume that $S$ intersects a subvariety of the form $V(I;N,N')$, for some $N,N'\in\N$, at a Hodge generic point of that subvariety and that \Cref{conjsupersingLT} holds. Then the set 
\begin{equation}\label{eq:zpmainy(1)eq} 
\Sha(S;I):=\{ s\in S(\C):\exists j\notin I\text{, } \exists M,M',M''\in \N\text { with } s\in V(I,j;M,M',M'') \} 
\end{equation}
 is finite. 
	\end{theorem}
	
\begin{remarks} 
$1$. By a ``Hodge generic point'' of a special subvariety $Z\subset Y(1)^n$ we mean a point that does not belong to any special subvariety of $Y(1)^n$ of dimension smaller than that of $Z$. The special subvarieties of $Y(1)^n$ are defined by equations of the form $\Phi_M(x_i,x_j)=0$ and equations of the form $x_i=j_0$, where $j_0$ is a fixed singular modulus. 
	
$2$. The Zilber--Pink conjecture predicts, in particular, that there are only finitely many points of intersection of $S$ with subvarieties of the form $V(I;N,N')$. Without additional hypotheses on $S$, this part of the conjecture remains open. For more on this, see the discussion in \Cref{section:summary}. 
\end{remarks}

The second finiteness result we establish concerns intersections with strongly special subvarieties corresponding to two disjoint triples of indices. 
\begin{theorem}\label{thmzpiny(1)n2} Let $S\subset Y(1)^n$, with $n\geq 6$, be a smooth irreducible curve defined over $\bar{\Q}$ that is not contained in a proper special subvariety of $Y(1)^n$. Let $I$ and $J$ be disjoint subsets of $\{1\ld n\}$ with $\# I=\# J=3$. 
	
Assume that $S$ intersects a subvariety of the form $V(I,J;M_0,M_0',N_0,N_0')$ at a Hodge generic point of that subvariety and that \Cref{conjsupersingLT} holds. Then the set 
\begin{equation}\label{eq:zpmainy(1)eq2} 
\Sha(S;I,J):=\{ s\in S(\C): \exists M,M',N,N'\in \N\text { with } s\in V(I,J;M,M',N,N') \} 
\end{equation} 
is finite. 
\end{theorem}

The arithmetic input used in both theorems is the following weak form of the Lang--Trotter conjecture for pairs of elliptic curves.
\begin{conj}[Conjecture on simultaneous supersingular reduction; \cite{dawpap}]\label{conjsupersingLT}
	Let $E$ and $E'$ be two elliptic curves defined over a number field $K$. Assume that $E$ and $E'$ do not acquire CM and are not isogenous to one another over $\bar{\Q}$. Consider the set
	\begin{equation}
		\Sigma(E,E')_{\ssing}:=\{v\in \Sigma_{K,f}: E\text{ and }E'\text{ are supersingular modulo }v\},
	\end{equation}
	and, for $x\geq 3$, define
	\begin{equation}
		\pi_{E,E'}(x):=\#\{p\leq x: \exists v\in \Sigma(E,E')_{\ssing} \text{ with } v|p\}.
	\end{equation}
	Then there exists a polynomial $P_{E,E'}(U,V)$, whose degree and coefficients depend only on $E$ and $E'$, such that, for all sufficiently large $x$,
	\[
	\pi_{E,E'}(x)\leq P_{E,E'}([K:\Q],\log\log x).
	\]
\end{conj}

\begin{remark}
	When $K=\Q$, \Cref{conjsupersingLT} is a slightly weaker form of a special case of the Lang--Trotter conjecture for pairs of elliptic curves. For pairs of elliptic curves defined over $\Q$, this conjecture first appears as Remark $2$ on page $37$ of the work of S. Lang and H. Trotter \cite{langtrotter}. 
	
	The natural expectation, based on heuristics from the Sato--Tate conjecture, is that $\pi_{E,E'}(x)$ behaves asymptotically like $C_{E,E',K}\log\log x$, as $x\rightarrow\infty$, for some positive constant $C_{E,E',K}$ depending on $E$, $E'$, and $K$. In this sense, \Cref{conjsupersingLT} may be viewed as a slightly weaker version of Lang--Trotter in this setting. For further details, see \cite{dawpap}.
	
	The above conjecture seems to have first been studied over $\Q$ by E. Fouvry and M. R. Murty; see \cite{fouvrymurty}, where they prove an average version. A conjectural description, again for pairs over $\Q$, of the aforementioned constant $C_{E,E'}$ is provided by A. Akbary and J. Parks in \cite{akbaryparks}, who describe a probabilistic model in the spirit of Lang--Trotter \cite{langtrotter}.
\end{remark}

Habegger and Pila \cite{habeggerpila1} studied Zilber--Pink for curves in $Y(1)^n$ via the Pila--Zannier strategy, and established the part of the conjecture involving special subvarieties that are not strongly special. Further work on Zilber--Pink in powers of the modular curve \cite{daworr4,daworrpap,papaszpy1} has produced several new cases in which the required height bounds are obtained through Andr\'e's $G$-function method. In these works, the contribution of finite places to the height is controlled using the way in which the curve meets the boundary $X(1)^n\backslash Y(1)^n$. 

The present paper is complementary to this circle of ideas: the finite-place contribution is controlled by isolating a different arithmetic obstruction, namely simultaneous supersingular reduction, and showing that Lang--Trotter-type sparsity for this obstruction supplies the height bounds needed for the finiteness results above. A related use of Lang--Trotter-type sparsity in unlikely intersections appears in joint work of the author with C. Daw in the setting of $\mathcal A_3$ \cite{dawpap}. The relevant Zilber--Pink background and the reduction to height bounds are recalled in \Cref{section:summary}.

\subsection{Outline of the paper} 

In \Cref{section:gfunsmethod} we recall the relevant Zilber--Pink background, explain the reduction to height bounds, and set up the $G$-functions method used in the proof. In \Cref{section:proofsandmore} we prove the height bounds needed for \Cref{thmzpiny(1)n,thmzpiny(1)n2}. We conclude in \Cref{section:speculation} by discussing the broader mechanism suggested by the proof, relating Lang--Trotter-type sparsity, supersingular reduction, and the $G$-function approach to unlikely intersections.

\subsection{Notation}
Given two elliptic curves $E$ and $E'$ over a field, we write $E\sim E'$ to mean that $E$ and $E'$ are isogenous over an algebraic closure of that field.

Let $\vec{Y}:=(y_1\ld y_N)\in K[[x]]^N$ be a family of power series, where $K$ is a number field, and let $v\in \Sigma_K$ be a place of $K$. We write $R_v(y_j)$ for the $v$-adic radius of convergence of $y_j$, and set
\[
R_v(\vec{Y}):=\min_j R_v(y_j).
\]

For such a place $v$, we write $\iota_v:K\hookrightarrow \C_v$ for the associated embedding, where $\C_v$ denotes either $\C$ or $\C_p$ according as $v$ is archimedean or non-archimedean. Finally, if $y(x)=\Sum{n=0}{\infty}a_n x^n\in K[[x]]$, we write
\[
\iota_v(y(x)):=\Sum{n=0}{\infty}\iota_v(a_n)x^n
\]
for the corresponding power series in $\C_v[[x]]$.

\section{Reduction to the $G$-functions method}\label{section:gfunsmethod} 

The purpose of this section is to pass from the finiteness statements of \Cref{thmzpiny(1)n,thmzpiny(1)n2} to the height bounds proved later by the $G$-functions method of Yves Andr\'e. We first recall the relevant form of the Zilber--Pink conjecture for curves in $Y(1)^n$ and the Habegger--Pila reduction to height bounds. We then state the two height bounds needed in the present paper in \Cref{section:theheightbounds}. In \Cref{section:andrebomb} we recall the Andr\'e--Bombieri Hasse principle. Finally, in \Cref{section:backgroundgfuns}, we describe the construction associating a family of $G$-functions to a tuple of elliptic schemes over a curve defined over a number field.

\subsection{Background on Zilber--Pink in $Y(1)^n$}\label{section:summary}

For curves in $Y(1)^n$, the Zilber--Pink conjecture takes the following form.
 \begin{conj}[Zilber--Pink for curves in $Y(1)^n$]\label{zilberpinkconj} Let $S\subset Y(1)^n$ be an irreducible curve that is not contained in a proper special subvariety of $Y(1)^n$. Then the set 
 \begin{equation*} 
S(\C)\cap \displaystyle\bigcup_{\cd Z\geq 2}Z(\C)
\end{equation*} 
is finite, where the union ranges over all special subvarieties of $Y(1)^n$ of codimension at least $2$. \end{conj}

The above problem was first studied by P. Habegger and J. Pila in \cite{habeggerpila1} via the Pila--Zannier strategy. The special subvarieties of $Y(1)^n$ are defined by systems of equations of the form $x_i=j_0$, for a fixed singular modulus $j_0$, and $\Phi_M(x_i,x_j)=0$, where $x_1,\ldots,x_n$ denote the coordinates of $Y(1)^n$. Following Habegger--Pila, we call a special subvariety \textbf{strongly special} if it can be defined without using equations of the form $x_i=j_0$. 

Habegger and Pila prove the following part of \Cref{zilberpinkconj}. \begin{theorem}[Habegger--Pila \cite{habeggerpila1}] Let $S$ be as in \Cref{zilberpinkconj} and assume that it is also defined over $\bar{\Q}$. Then the set \begin{equation} S(\C)\cap \displaystyle\bigcup_{\underset{Z\text{ not strongly special }}{\cd Z\geq 2}}Z(\C) \end{equation} is finite. \end{theorem}

The open part of \Cref{zilberpinkconj} is therefore the finiteness of the strongly atypical locus of $S$, namely the set 
\begin{equation}
	\atyp(S)_{\strspec}:=S(\C)\cap \displaystyle\bigcup_{\underset{Z\text{ strongly special }}{\cd Z\geq 2}}Z(\C).
\end{equation}

In this direction, the conjecture is known for curves defined over $\C\setminus \bar{\Q}$ by work of J. Pila \cite{pilagen}. For curves defined over $\bar{\Q}$, the known results impose additional hypotheses on $S$. 
\begin{theorem} Let $S\subset Y(1)^n$ be as in \Cref{zilberpinkconj} and assume that $S$ is defined over $\bar{\Q}$. Then the set $\atyp(S)_{\strspec}$ is finite in each of the following cases: 
	\begin{enumerate} \item $S$ is asymmetric \cite{habeggerpila1}; 
		\item $S$ intersects the boundary of the compactification of $Y(1)^n$ at the point $\infty^n$ \cite{daworr4}; 
		\item $S$ intersects the boundary of the compactification of $Y(1)^n$ at a point all of whose coordinates are either $\infty$ or singular moduli \cite{papaszpy1}; 
		\item $n=3$ and $S$ intersects the boundary of the compactification of $Y(1)^n$ at a ``modular point'' \cite{daworrpap}. \end{enumerate} \end{theorem}

The results of this paper give two finiteness statements, under \Cref{conjsupersingLT}, for explicit families of strongly special intersections. More precisely, \Cref{thmzpiny(1)n} proves finiteness for the intersections in which one fixed triple is enlarged by an additional modular relation, while \Cref{thmzpiny(1)n2} proves finiteness for intersections involving two disjoint triples. In particular, \Cref{thmzpiny(1)n2} establishes the finiteness of the set of points of intersection of $S$ with the countably infinite union of strongly special subvarieties 
\begin{equation*} 
	\displaystyle\bigcup_{M,M',N,N'\in \N}V(I,J;M,M',N,N')(\C). 
\end{equation*}

For general background on the Zilber--Pink conjecture, we refer the reader to \cite{pilabook}. For the specific case of $Y(1)^n$, see \cite{daworr4,daworrpap}.

\subsection{From Zilber--Pink to height bounds}\label{section:fromzpheight} 

Following \cite{habeggerpila1}, \Cref{zilberpinkconj} is reduced to lower bounds for the sizes of Galois orbits of points of intersection of a curve $S$ with special subvarieties of codimension at least $2$. 

Habegger and Pila further show that, using the isogeny estimates of Masser--W\"ustholz \cite{masserwuisogellcurves}, these lower bounds follow from suitable upper bounds for the Weil heights of the points in question.
\subsection{The height bounds}\label{section:theheightbounds} 

For the family of intersections appearing in \Cref{thmzpiny(1)n}, the required height bound, established in \Cref{section:htboundproof1}, is the following.
\begin{theorem}\label{heightboundintro}
	Let $S\subset Y(1)^n$ and $I\subset\{1,\ldots,n\}$ be as in \Cref{thmzpiny(1)n}, and let $s_0\in S(\bar{\Q})$ be a Hodge generic point of intersection as in the hypothesis of \Cref{thmzpiny(1)n}. Let $h$ be a Weil height on $S$ and assume that \Cref{conjsupersingLT} holds. 
	
	Then there exist positive constants $c_0$, $c_1$ such that for all points $s$ in the set $\Sha(S;I)$ we have
	\begin{equation}h(s)\leq c_0 [\Q(s):\Q]^{c_1}.\end{equation}\end{theorem}

Similarly, for the family of intersections appearing in \Cref{thmzpiny(1)n2}, the corresponding height bound, established in \Cref{section:htboundproof2}, is the following. 
\begin{theorem}\label{heightbound2intro} 
Let $S\subset Y(1)^n$ and $I,J\subset\{1,\ldots,n\}$ be as in \Cref{thmzpiny(1)n2}, and let $s_0\in S(\bar{\Q})$ be a Hodge generic point of intersection as in the hypothesis of \Cref{thmzpiny(1)n2}. Let $h$ be a Weil height on $S$ and assume that \Cref{conjsupersingLT} holds. 
	
	Then there exist positive constants $c_0$, $c_1$ such that for all points $s$ in the set $\Sha(S;I,J)$ we have \begin{equation} h(s)\leq c_0 [\Q(s):\Q]^{c_1}.\end{equation}\end{theorem}

\subsection{From curves in $Y(1)^n$ to families}\label{section:heightsinfamilies}
We reformulate the height bounds of \Cref{section:theheightbounds} in terms of tuples of elliptic schemes over a curve. This is the geometric setting in which we apply the $G$-functions method.

After replacing the curve by a finite cover when necessary, a curve in $Y(1)^n$ gives rise to a corresponding tuple of one-parameter families of elliptic curves. In this geometric setting, \Cref{heightboundintro} takes the following form.
\begin{theorem}\label{mainhtbound}
	Let $\CE_1,\ldots,\CE_n\rightarrow S$ be one-parameter families of elliptic curves defined over a number field $K$, and write
\[
f:X:= \CE_1\times_{S}\ldots\times_{S}\CE_n \rightarrow S
\]
for their fiber product. Let $h$ be a Weil height on $S$.

Assume there exist indices $i_1<i_2<i_3$ and a point $s_{0} \in S(K)$ for which $X_{s_{0}} = E_{1}\times_K\ldots\times_K E_{n}$ is such that: 
\begin{enumerate} 
	\item the elliptic curves $E_{i_k}$, $k=2,3$, are isogenous to $A_1:=E_{i_1}$,
	\item $A_1$ does not acquire CM over $\bar{\Q}$, and
	\item for all $j\notin\{i_1,i_2,i_3\}$, $E_j$ is not isogenous to $A_1$ and does not acquire CM over $\bar{\Q}$. 
	\end{enumerate}
	Assume also that the induced morphism $S \rightarrow Y(1)^n$ has image a Hodge generic curve and \Cref{conjsupersingLT} holds for the pair $(A_{1},E_{j})$, for all $j\notin\{i_1,i_2,i_3\}$. 
	
	Then there exist constants $c_{1}, c_{2}>0$ such that $h(s) \leq c_{1} \cdot[K(s): \Q]^{c_{2}}$ for all $s$ in the set
	\[
	\left\{s \in S(\bar{\Q}): \CE_{i_1,s}\sim\CE_{i_2,s} \sim \CE_{i_3,s}\sim\CE_{j,s} \text{ for some }j\notin\{i_1,i_2,i_3\} \right\}.
	\]
\end{theorem}

Similarly, in this geometric setting, \Cref{heightbound2intro} takes the following form.
\begin{theorem}\label{mainhtbound2}
	Let $\CE_1,\ldots,\CE_n\rightarrow S$ be one-parameter families of elliptic curves defined over a number field $K$, and write
	\[
	f:X:= \CE_1\times_{S}\ldots\times_{S}\CE_n \rightarrow S
	\]
	for their fiber product. Let $h$ be a Weil height on $S$.
	
	Assume there exist two disjoint triples of indices $i_1<i_2<i_3$, $j_1<j_2<j_3$, and a point $s_{0} \in S(K)$ for which $X_{s_{0}} = E_{1}\times_K\ldots\times_K E_{n}$ is such that:
	\begin{enumerate}
		\item the elliptic curves $E_{i_k}$, $k=2,3$, are isogenous to $A_1:=E_{i_1}$,
		\item the elliptic curves $E_{j_k}$, $k=2,3$, are isogenous to $A_2:=E_{j_1}$, and
		\item neither $A_1$ nor $A_2$ acquires CM over $\bar{\Q}$, and $A_1$ and $A_2$ are not isogenous to each other over $\bar{\Q}$.
	\end{enumerate}
	Assume also that the induced morphism $S \rightarrow Y(1)^n$ has image a Hodge generic curve and \Cref{conjsupersingLT} holds for the pair $(A_{1},A_{2})$. 
	
	Then there exist constants $c_{1}, c_{2}>0$ such that $h(s) \leq c_{1} \cdot[K(s): \Q]^{c_{2}}$ for all $s$ in the set
	\[
	\left\{s \in S(\bar{\Q}): \CE_{j_1,s}\sim\CE_{j_2,s} \sim \CE_{j_3,s}\text{ and }\CE_{i_1,s}\sim\CE_{i_2,s} \sim \CE_{i_3,s} \right\}.
	\]
\end{theorem}

The proofs of \Cref{mainhtbound,mainhtbound2} are given in \Cref{section:htboundproof1,section:htboundproof2}, respectively.

\subsection{$G$-functions and the Andr\'e--Bombieri Hasse principle}\label{section:andrebomb}
The notion of a $G$-function was introduced by C. L. Siegel in \cite{siegel}. We recall the form of the Andr\'e--Bombieri Hasse principle for values of $G$-functions used below.

\begin{definition}[\cite{siegel}]
	Let $K$ be a number field and $y(x)=\displaystyle\sum_{n=0}^{\infty}a_nx^n\in K[[x]]$. The power series $y$ is a $G$-function if 
	\begin{enumerate}
		\item $y$ is a solution of a linear homogeneous differential equation with coefficients in $K(x)$,
		\item for all $v\in \Sigma_{K,\infty}$, $\iota_v(y)$ defines an analytic function in a neighborhood of $0$, and 
		\item there exists $C>0$ such that, for all $n\in \N$, there exists $d_n\in \N$ with $d_n\leq C^{n+1}$ and $d_na_m\in \mathcal{O}_K$ for all $m\leq n$.
	\end{enumerate}
\end{definition}

Bombieri initiated the study of values of $G$-functions at algebraic numbers in \cite{bombg}. This line of work was subsequently developed by Andr\'e in \cite{andre1989g}. To make precise the notion of a relation among the values of $G$-functions at an algebraic number $\xi$, Andr\'e introduced the following definitions.
\begin{definition}[\cite{andre1989g}]\label{defrelations}
	Let $\vec{Y}\in K[[x]]^N$ be a collection of $G$-functions, let $R\in L[X_1\ld X_N]$ be a homogeneous polynomial with coefficients in some number field $L$, and let $\xi\in \bar{\Q}$. We say that $R$ defines a relation among the values $\vec{Y}(\xi)$ that is 
	\begin{enumerate}
		\item \textbf{global}, if $\iota_v(R)(\iota_v(\vec{Y})(\iota_v(\xi)))=0$ for all places $v\in \Sigma_{LK(\xi)}$ for which $|\iota_v(\xi)|_v<r_v(\vec{Y}):=\min\{1, R_v(\vec{Y})\}$, and 
		
		\item \textbf{non-trivial}, if there exists no $\tilde{R}\in \bar{\Q}[X_1\ld X_N,x]$ with $\tilde{R}(\underline{X},\xi)=R(\underline{X})$ and $\tilde{R}(\vec{Y},x)=0$ as power series.
	\end{enumerate}
\end{definition}

With these definitions, the Andr\'e--Bombieri Hasse principle for values of $G$-functions is the following.
\begin{theorem}[\cite{bombg,andre1989g}]\label{thmandbomb}
	Let $\vec{Y}\in K[[x]]^N$ be a collection of $G$-functions. For $\delta>0$ consider the set 
	\begin{equation}
		\Sha_{\delta}(\vec{Y}):=\{\xi\in\bar{\Q}:\exists\text{ non-trivial, global relation of degree }\leq\delta \text{ among the } \vec{Y}(\xi)\}.
	\end{equation}
	
	Then there exist positive, effectively computable constants $c_1$ and $c_2$ such that $h(\xi)\leq c_1\delta^{c_2}$ for all $\xi\in \Sha_{\delta}(\vec{Y})$.
\end{theorem}

\subsection{Families of elliptic curves and $G$-functions}\label{section:backgroundgfuns}
We recall the construction of the $G$-functions attached to one-parameter families of elliptic curves. After the standard reductions recalled below, let $f:\CE\rightarrow S$ be a family of elliptic curves over a smooth, geometrically irreducible curve $S$ defined over a number field $K$, and fix a point $s_0\in S(K)$.

For the height bounds considered here, we may replace $K$ by a finite extension and $S$ by a finite cover whenever necessary. This allows us to impose standard auxiliary assumptions on the pair $(f:\CE\rightarrow S,s_0)$. For this reduction, see the discussion in \S 2.2 of \cite{papaszpy1}.

In particular, we may assume that the fiber $\CE_{s_0}$ satisfies the reduction properties required below, such as everywhere semistable reduction; see Lemma 2.4 of \cite{papaszpy1}. We may also assume that there exists a local parameter $x\in K(S)$ at $s_0$ with the properties needed to center the $G$-functions at $s_0$. For these properties of $x$, see Lemma 5.1 of \cite{daworr4}.

The construction of Daw--Orr in \S 5 of \cite{daworr4} associates to the data $(\CE\rightarrow S,s_0,x)$ a tuple of families
\[
\CE_1\times_S\ldots\times_S \CE_{\Lambda}\rightarrow S
\]
by taking suitable twists of the original family. The construction and the required choices are treated in \cite{daworr4,papaspadicpart1,daworrpap}. From this point on, we work in the simplified setting in which the chosen local parameter $x$ has a simple zero only at $s_0$; in this case the above construction has $\Lambda=1$.

Under the above assumptions, one associates to the pair $(\CE\rightarrow S,s_0)$ a matrix $Y_\CE\in M_{2}(K[[x]])$ whose entries are $G$-functions centered at $s_0$. It may be chosen so that:
 \begin{enumerate} 
	\item $Y_{\CE}\in \SL_2(K[[x]])$, and 
	\item for any $s\in S(\bar{\Q})$ and any $v\in \Sigma_{K(s),f}$, we have $|x(s)|_v<r_v(Y_{\CE}):=\min\{1,R_v(Y_{\CE})\}$ if and only if $\CE_s$ and $\CE_{s_0}$ have the same reduction modulo $v$. 
	\end{enumerate} 
	For the construction of $G$-functions from one-parameter families of abelian varieties, see \cite{andre1989g}. The normalization above is obtained by choosing an appropriate basis of sections of $H^1_{dR}(\CE/S)$; see \S 3 of \cite{papaspadicpart1}.

Given a tuple $\CE_1\times_S\ldots\times_S \CE_n\rightarrow S$ in the above sense, we associate to each $\CE_j$ its matrix $Y_j:=Y_{\CE_j}$. The collection of $G$-functions used below is denoted by \[ Y_G:=\{Y_1,\ldots,Y_n\}. \] We also set $r_v(Y_G):=\min_j r_v(Y_j)$. With this notation, for any point $s\in S(\bar{\Q})$ and any place $v$ of $K(s)$, we say that $s$ is $v$\textbf{-adically close to }$s_0$ if $|x(s)|_v<r_v(Y_G)$.
\section{Height bounds}\label{section:proofsandmore}

We prove the two height bounds stated in \Cref{section:heightsinfamilies}.

\subsection{Local factors}\label{section:localfactors}
The construction of global and non-trivial relations, in the terminology of \Cref{defrelations}, proceeds by first constructing relations at each place of proximity and then multiplying them together. We refer to these local relations as the local factors of the relation. 

In the present setting, the required local factors were constructed in joint work of the author with C. Daw and M. Orr \cite{daworrpap}. The relevant statement is Proposition~3.6 of \cite{daworrpap}, which we use in the following form, adapted to the notation of \Cref{section:backgroundgfuns}.
\begin{prop}[\cite{daworrpap}]\label{localfactorsprop}
Let $f_k:\CE_k\rightarrow S$, for $1\leq k\leq 3$, be $1$-parameter families of elliptic curves defined over a number field $K$, and let $s_0\in S(K)$. Let $Y_{k}:=(y_{i,j,k})$, for $1\leq k\leq 3$, be the $2\times 2$ matrices of $G$-functions associated to the $f_k$ and centered at $s_0$.

Assume that $s_0$ is such that the fibers $E_i:=\CE_{i,s_0}$, $1\leq i\leq 3$, are all isogenous, and let $s\in S(\bar{\Q})$ be a point for which the fibers $\CE_{j,s}$, $1\leq j\leq 3$, are all isogenous.
\begin{enumerate}
	\item For each $v\in \Sigma_{K(s)}$ for which $s$ is $v$-adically close to $s_0$ and, if $v$ is finite, $E_1$ does not have ordinary reduction modulo $v$, there exists $R_{s,v}\in K(s)[X_{i,j,k}:1\leq i,j\leq 2, 1\leq k\leq 3]$ for which the following hold:
\begin{enumerate}[label=\roman*.]
	\item $R_{s,v}$ is homogeneous with $\deg(R_{s,v})\leq 2$,
\item $\iota_v\left(R_{s,v}(Y_k(x(s));1\leq k\leq 3)\right)=0$, and
	\item $R_{s,v}$ is not in the ideal $(\det(X_{i,j,k})-1;1\leq k\leq 3)$.
\end{enumerate}

\item There exists a polynomial $R_{s,\ord}\in K(s)[X_{i,j,k}:1\leq i,j\leq 2, 1\leq k\leq 3]$ such that the following hold:
\begin{enumerate}[label=\roman*.]
	\item $R_{s,\ord}$ is homogeneous with $\deg(R_{s,\ord})\leq 4$,
    \item $\iota_v\left(R_{s,\ord}(Y_k(x(s));1\leq k\leq 3)\right)=0$ for all $v\in \Sigma_{K(s),f}$ for which $s$ is $v$-adically close to $s_0$ and $E_1$ is ordinary modulo $v$, and
	\item $R_{s,\ord}$ is not in the ideal $(\det(X_{i,j,k})-1;1\leq k\leq 3)$.
\end{enumerate}\end{enumerate}
\end{prop}

\subsection{Proof of \Cref{mainhtbound}}\label{section:htboundproof1}
	We work in the simplified setting fixed in \Cref{section:backgroundgfuns}. The reductions leading to this setting, and the corresponding arguments in greater generality, are treated in \cite{daworr4,daworrpap,papaspadicpart1}; see in particular the proof of Proposition~5.1 of \cite{papaspadicpart1}. The proof below follows the $G$-function height-bound strategy developed in \cite{daworr4,daworrpap,dawpap}.

By relabeling the families, we may assume that the distinguished triple of indices is $I=\{1,2,3\}$. Since there are only finitely many possible choices of the additional index $j\notin I$, it suffices to prove the desired bound after fixing one such index, enlarging the constants at the end if necessary. We therefore fix a point $s\in S(\bar{\Q})$ in the set appearing in \Cref{mainhtbound} and, after relabeling the remaining factors, assume that the corresponding index is $j=4$, so that $\CE_{1,s},\ldots,\CE_{4,s}$ are all isogenous. Set \begin{equation}\label{eq:primesofprox} 
		\Sigma(s):=\left\{v \in \Sigma_{K(s)}: s \text{ is } v\text{-adically close to } s_{0}\right\}. 
		\end{equation}
	
	By the definition of $v$-adic proximity, if $v\in \Sigma(s)$ is finite, then each of the fibers $\CE_{j,s}$, $1\leq j\leq 4$, has the same reduction modulo $v$ as the corresponding fiber $E_j=\CE_{j,s_0}$. Since the fibers $\CE_{i,s}$, $1\leq i\leq 4$, are all isogenous, they have the same reduction type at $v$: either they all have ordinary reduction, they all have supersingular reduction, or they all have bad reduction. This follows from the fact that isogenies preserve good reduction and, in good reduction, preserve the ordinary/supersingular alternative; see for example Theorem~V.3.1 of \cite{silvermanell}.

	We decompose $\Sigma(s)$ as follows: 
	\begin{enumerate} 
		\item $\Sigma(s)_{\ord}$ is the subset of finite places $v$ for which $E_1$ and $E_4$ both have ordinary reduction modulo $v$, 
		\item $\Sigma(s)_{\ssing}$ is the subset of finite places $v$ for which $E_1$ and $E_4$ both have supersingular reduction modulo $v$, 
		\item $\Sigma(s)_{\bad}$ is the subset of finite places $v$ for which $E_1$ and $E_4$ both have bad reduction modulo $v$, and 
		\item $\Sigma(s)_{\infty}$ is the subset of infinite places in $\Sigma(s)$. 
		\end{enumerate} By the preceding paragraph, $\Sigma(s)$ is the disjoint union of these four sets.
	
We apply \Cref{localfactorsprop} to the first three elliptic schemes $\CE_k\rightarrow S$, $1\leq k\leq 3$. Let $R_{s,\ord}$ and $R_{s,v}$ denote the polynomials obtained from \Cref{localfactorsprop}. Define \begin{equation}\label{eq:globalrelation} R_{s}:=\left(\prod_{v \in \Sigma(s)_{\infty}} R_{s, v}\right) \cdot\left(\prod_{v \in \Sigma(s)_{\ssing}} R_{s, v}\right) \cdot\left(\prod_{v \in \Sigma(s)_{\bad}} R_{s, v}\right) \cdot R_{s, \ord}, \end{equation} with the convention that an empty product is equal to $1$. With this convention, $R_s$ defines the required global relation among the values of the relevant entries of $Y_G$ at $x(s)$.

The relation defined by \eqref{eq:globalrelation} is non-trivial since the ideal
\[
(\det(X_{i,j,k})-1:1\leq k\leq 3)
\]
is prime, and none of the local factors $R_{s,v}$ and $R_{s,\ord}$ lies in this ideal. For the non-triviality criterion, see \S 4.5.6 of \cite{daworrpap}. It remains to bound $\deg(R_s)$ strongly enough to deduce the desired height bound from \Cref{thmandbomb}.

First consider $R_{s,\infty}:=\prod_{v \in \Sigma(s)_{\infty}} R_{s,v}$. By \Cref{localfactorsprop}, each factor $R_{s,v}$ has degree at most $2$, and therefore $\deg(R_{s,\infty})\leq 2[K(s):\Q]$. Next, let $c_{\bad}$ denote the number of finite places $w\in \Sigma_{K,f}$ at which $E_1$ has bad reduction. This number is finite and independent of $s$. For $R_{s,\bad}:=\prod_{v \in \Sigma(s)_{\bad}} R_{s,v}$, it follows that $\deg(R_{s,\bad})\leq 2c_{\bad}[K(s):\Q]$.

It remains to bound the degree of $R_{s,\ssing}:=\prod_{v \in \Sigma(s)_{\ssing}} R_{s,v}$. This is where \Cref{conjsupersingLT} enters. Define \begin{equation}\label{eq:ssingsetoverK} \Sigma(s)_{\ssing,K}:=\{ w\in\Sigma_{K,f}:\exists v\in \Sigma(s)_{\ssing} \text{ with }v|w\}, \end{equation} and set $\pi_K(s):=\# \Sigma(s)_{\ssing,K}$ and $\pi(s):=\#\Sigma(s)_{\ssing}$. Since each $R_{s,v}$ has degree at most $2$, we have $\deg(R_{s,\ssing})\leq 2\pi(s)\leq 2[K(s):\Q]\pi_K(s)$.

Combining the preceding bounds gives \begin{equation}\label{eq:proof1degreebound1} 
	\deg(R_s)\leq 2(c_{\bad}+1+\pi_K(s))[K(s):\Q]+4. \end{equation} 

It remains to control the supersingular contribution, measured here by $\pi_K(s)$. This is achieved by the following bound.
 \begin{claim}\label{claimusinglangtrotter} There exist positive constants $C_0$, $C_1$, and $D$ such that, for every such $s$, \begin{enumerate} \item either $h(x(s))\leq C_0$, or \item $\pi_K(s)\leq C_1([K(s):\Q]^{D}+(\log h(x(s)))^D)$. \end{enumerate} \end{claim}
\begin{proof}[Proof of \Cref{claimusinglangtrotter}]
Set \[ 
A(s):=\{p\in \Sigma_{\Q,f}: \exists v\in \Sigma(s)_f \text{ with } v|p\} 
\] 
to be the set of rational primes lying below finite places of proximity. We also set 
\[ 
\Sigma(s)_{\ssing,\Q}:=\{p\in \Sigma_{\Q,f}: \exists w\in \Sigma(s)_{\ssing,K} \text{ with } w|p\}.
 \] 
 Finally, let $\pi_{\Q}(s):=\#\Sigma(s)_{\ssing,\Q}$, and note that $\pi_K(s)\leq [K:\Q]\pi_{\Q}(s)$.

By the finite-place contribution to the logarithmic height of $x(s)$, we have \[ \sum_{p\in \Sigma(s)_{\ssing,\Q}} \log p \leq \sum_{p\in A(s)} \log p \leq [K(s):\Q] h(x(s)). \] In particular, if $p\in \Sigma(s)_{\ssing,\Q}$, then \[ p\leq \exp([K(s):\Q]h(x(s))). \]

Set \[ \alpha(s):=[K(s):\Q]h(x(s)). \] In the notation of \Cref{conjsupersingLT}, the preceding bound gives 
\[ 
\pi_{\Q}(s)\leq \pi_{E_1,E_4}(\exp(\alpha(s))). 
\] 
By the assumptions on the pair $(E_1,E_4)$, \Cref{conjsupersingLT} gives, provided $\alpha(s)\geq C_{\mathrm{LT}}$ for a constant $C_{\mathrm{LT}}$ depending only on $E_1$ and $E_4$, 
\[ 
\pi_{\Q}(s)\leq P_{E_1,E_4}([K(s):\Q],\log \alpha(s)), 
\]
where $P_{E_1,E_4}(X,Y)$ is a polynomial whose degree and coefficients depend only on $E_1$ and $E_4$.

After choosing $D$ sufficiently large, depending only on the degrees of the monomials of $P_{E_1,E_4}$, we may choose $C_1$ sufficiently large so that
\[
 P_{E_1,E_4}([K(s):\Q],\log \alpha(s)) \leq C_1\left([K(s):\Q]^D+(\log h(x(s)))^D\right), 
 \]
 provided $\alpha(s)\geq C_{\mathrm{LT}}$.

	If either $\alpha(s)<C_{\mathrm{LT}}$ or $h(x(s))$ lies in a bounded range, then the first alternative of the claim holds, for $C_0$ chosen sufficiently large depending only on $E_1$ and $E_4$. Otherwise the preceding bound gives the second alternative, after increasing $C_1$ if necessary, using $\pi_K(s)\leq [K:\Q]\pi_{\Q}(s)$.
	\end{proof}
   
   Increasing $C_0$ in \Cref{claimusinglangtrotter} if necessary, we may assume that $C_0\geq e$. If the first alternative in \Cref{claimusinglangtrotter} holds, then $h(x(s))$ is already bounded independently of $s$, and the desired conclusion follows after the height comparison at the end of the proof. We may therefore assume that the first alternative does not hold. Thus $h(x(s))>C_0$, and in particular $\log h(x(s))\geq 1$.
   
 	Combining \eqref{eq:proof1degreebound1} with the bound for $\pi_K(s)$ in \Cref{claimusinglangtrotter}, we obtain positive constants $C_4$ and $C_5$, independent of $s$, such that \[ \deg(R_s)\leq C_4[K(s):\Q]^{D+1} +C_5[K(s):\Q](\log h(x(s)))^D. \] Since $\log h(x(s))\geq 1$, after increasing $C_4$ if necessary, this gives \begin{equation}\label{eq:proof1degreebound2} \deg(R_s)\leq C_4[K(s):\Q]^{D+1}(\log h(x(s)))^D. \end{equation}
 	
   Since $R_s$ defines a global non-trivial relation among the values of $Y_G$ at $x(s)$, \Cref{thmandbomb} gives positive constants $c_1$ and $c_2$, independent of $s$, such that \begin{equation}\label{eq:htbound1} h(x(s)) \leqslant c_{1} (\deg R_{s})^{c_{2}}. \end{equation}
   
   Combining this with \eqref{eq:proof1degreebound2}, and using the present assumption that $h(x(s))>C_0$, we obtain constants $c'_1>0$ and $c'_2>0$ such that \begin{equation*} \frac{h(x(s))}{(\log h(x(s)))^{D\cdot c_2}}\leq c'_1[K(s):\Q]^{c'_2}. \end{equation*} Absorbing the logarithmic factor into the constants, we obtain \begin{equation}\label{eq:proof1heightlogbound}  h(x(s))\leq c''_1[K(s):\Q]^{c''_2}.\end{equation}
   
   By the classical properties of the Weil height machine, see for example Theorem~B.3.2 of \cite{hindrysilverman}, applied to the finite morphism induced by $x$ on the smooth projective model of $S$, the height $h(s)$ is bounded above by a linear function of $h(x(s))$. The desired height bound follows.
\subsection{Proof of \Cref{mainhtbound2}}\label{section:htboundproof2}

For simplicity we assume that $I=(i_1,i_2,i_3)=(1,2,3)$ and that $J=(j_1,j_2,j_3)=(4,5,6)$. We write $\Sigma(s)$, as in \eqref{eq:primesofprox}, for the set of places of proximity of $s$ and $s_0$. The proof is parallel to that of \Cref{mainhtbound}. We record the modifications needed in the present setting.

Fix a point $s$ as described in \Cref{mainhtbound2}. For each of the triples $I$ and $J$, we may apply \Cref{localfactorsprop} to obtain local factors for the global relation. We denote these by $R_{s,v}$ and $R_{s,\ord}$ when they are constructed for the triple $I$, and by $Q_{s,v}$ and $Q_{s,\ord}$, respectively, when they are constructed for the triple $J$.

Writing $E_1\times_K\ldots\times_K E_n$ for the fiber over $s_0$, we decompose $\Sigma(s)$ as follows:
\begin{enumerate}
	\item $\Sigma(s)_{\ord}$ is the subset of finite places $v$ for which at least one of $E_1$ and $E_4$ has ordinary reduction modulo $v$,
	\item $\Sigma(s)_{\ssing}$ is the subset of finite places $v$ for which both $E_1$ and $E_4$ have supersingular reduction modulo $v$,
	\item $\Sigma(s)_{\bad}$ is the subset of finite places $v$ for which neither $E_1$ nor $E_4$ has ordinary reduction modulo $v$, and at least one of them has bad reduction modulo $v$, and
	\item $\Sigma(s)_{\infty}$ is the subset of infinite places in $\Sigma(s)$.
\end{enumerate}

Consider
\begin{equation}\label{eq:globalrelation2} 
	R_{s}:=\left(\prod_{v \in \Sigma(s)_{\infty}} R_{s, v}\right) \cdot\left(\prod_{v \in \Sigma(s)_{\ssing}} R_{s, v}\right) \cdot\left(\prod_{v \in \Sigma(s)_{\bad}} R_{s, v}\right) \cdot R_{s, \ord}\cdot Q_{s,\ord}. 
\end{equation}
By construction, the polynomial $R_s$ defines a global relation among the values of the relevant entries of $Y_G$ at $x(s)$. Its non-triviality follows, as in the proof of \Cref{mainhtbound}, from the primeness of the ideal $(\det(X_{i,j,k})-1:1\leq k\leq n)$ and the fact that none of the local factors lies in this ideal.

Applying \Cref{thmandbomb}, we obtain \[ 
h(x(s))\leq c_1(\deg R_s)^{c_2}. 
\]
 It remains to bound $\deg(R_s)$. As in the proof of \Cref{mainhtbound}, the infinite places contribute at most $2[K(s):\Q]$, while the bad places contribute at most $2c_{\bad}[K(s):\Q]$, for a constant $c_{\bad}$ depending only on the number of places of bad reduction of $E_1$ and $E_4$. Finally, the two ordinary factors $R_{s,\ord}$ and $Q_{s,\ord}$ have degree at most $4$ each. Thus 
 \[
  \deg(R_s)\leq 2(c_{\bad}+1+\pi_K(s))[K(s):\Q]+8, 
  \] 
 where $\pi_K(s)$ denotes the number of places of $K$ below places in $\Sigma(s)_{\ssing}$. The argument of \Cref{claimusinglangtrotter}, applied to the pair $(E_1,E_4)$, bounds $\pi_K(s)$ in terms of $[K(s):\Q]$ and $\log h(x(s))$. The final absorption and height-comparison steps are then exactly as in the proof of \Cref{mainhtbound}.
\subsection{Lang--Trotter phenomena and unlikely intersections}\label{section:speculation}

Let $X$ be either $Y(1)^n$ or $\mathcal{A}_g$, the moduli space of principally polarized $g$-dimensional abelian varieties, and let $S\subset X$ be a curve defined over $\bar{\Q}$. We explain a general strategy, suggested by the preceding arguments, for proving Zilber--Pink-type statements for intersections of $S$ with special subvarieties of $X$ corresponding to abelian varieties with extra endomorphisms.

From the perspective of the $G$-functions method, a natural approach is to start from a point $s_0\in S$ lying on one such special subvariety and then bound the heights of the other points $s\in S(\bar{\Q})$ arising from intersections with special subvarieties of the same type. For example, when $X=\mathcal{A}_2$, one may start from a point $s_0$ corresponding to an abelian surface with quaternionic multiplication and seek a height bound for all other such points on $S$. This is the perspective taken in \Cref{thmzpiny(1)n,thmzpiny(1)n2}.

Let $Y_G$ denote the family of $G$-functions associated to $S$ and centered at $s_0$, as in \Cref{section:backgroundgfuns}. This line of inquiry was studied for $\mathcal{A}_2$ by the author in \cite{papaspadicpart1,papaspadicpart3} and for $Y(1)^3$ in joint work with C. Daw and M. Orr \cite{daworrpap}. These works isolate the first ingredient used here: at places of ordinary reduction of the central fiber over $s_0$, the local factor $R_{s,v}$ of the global relation $R_s$ may be chosen independently of $v$.

This phenomenon reduces the required height bound, via \Cref{thmandbomb}, to controlling the contribution to the global non-trivial relation $R_s$ coming from places of proximity at which the central fiber over $s_0$ has supersingular reduction. In the language of \Cref{section:htboundproof1}, these are the places of supersingular proximity of $s$ to $s_0$. 

It is natural, in light of the results in \cite{dawpap,daworrpap,papaspadicpart1,papaspadicpart2,papaspadicpart3}, to expect this mechanism to persist in greater generality. We separate this expectation into two features. 

\subsubsection{Independence of local factors at non-supersingular places} 

The independence of local factors at ordinary places in the cases discussed above rests on the action of Frobenius on crystalline cohomology. In the ordinary case, the Frobenius eigenvalues have distinct valuations, and the resulting slope decomposition reduces the number of algebraic quantities in the period relations that depend on the place $v$.

For higher-dimensional abelian varieties, reduction modulo $v$ is no longer governed simply by the ordinary/supersingular duality. Instead, the possible reduction types are reflected in finitely many Newton polygons, whose possible slopes depend on the dimension $g$ of the abelian varieties under consideration. This suggests that non-supersingular Newton polygons of a fixed type should contribute a single local factor independent of the place $v$, possibly depending on the codimension of the special subvariety with which $S$ is intersected. 

The second feature is the sparsity of the supersingular places of proximity, which is the aspect studied more directly here and in \cite{dawpap}. 

\subsubsection{Sparsity of supersingular places of proximity}

In all constructions of global relations mentioned above, the remaining contribution is the analogue of the factor 
\[
\prod_{v \in \Sigma(s)_{\ssing}} R_{s,v} 
\]
appearing in \eqref{eq:globalrelation}. Ideally, one would like to bound the degree of this factor polynomially in $[K(s):\Q]$ and sub-polynomially in $h(s)$.

In \cite{daworrpap,papaspadicpart1,papaspadicpart3}, the problem of bounding the number of such places was isolated in the study of intersections of a curve $S$, either in $\mathcal{A}_2$ or in $Y(1)^3$, with special subvarieties of the smallest codimension relevant to Zilber--Pink in these settings, namely codimension $2$ special subvarieties.

Considering intersections with special subvarieties of higher codimension, as is done here and in \cite{dawpap}, reduces the desired height bounds to a Lang--Trotter-type sparsity statement for pairs of elliptic curves defined over a number field $K$. As noted in the introduction, this conjecture for pairs has been studied before, at least over $\Q$; see \cite{akbaryparks,fouvrymurty}. The form of the conjecture used here, namely \Cref{conjsupersingLT}, is slightly weaker than the bounds predicted by Sato--Tate-type heuristics. For more on this point, we refer to \cite{dawpap}.

From the perspective of the $G$-functions method, the sparsity of supersingular places of proximity emerges as the natural number-theoretic obstruction in this approach to Zilber--Pink. Based on Sato--Tate-type heuristics, it is natural, at least in the settings considered here and in \cite{dawpap}, to expect that the places of supersingular reduction of the central fiber over $s_0$ are sparse when $s_0$ is sufficiently generic in the ambient moduli space. This is precisely the gap filled by \Cref{conjsupersingLT} in our setting. In contrast, the same heuristics no longer predict such sparsity for the places of supersingular reduction of the central fiber over $s_0$ when $s_0$ is not generic. For example, for CM elliptic curves, Deuring's criterion implies that, away from finitely many primes, supersingular reduction occurs precisely at the primes which are inert in the CM field. In particular, the supersingular primes have density $1/2$ among the rational primes.

\section*{Acknowledgments} 

The author thanks Gal Binyamini for suggesting that the techniques of \cite{dawpap} might be applied to the setting of $Y(1)^n$, for numerous helpful discussions concerning the results of this paper, and for his encouragement to prepare this manuscript. 

During the preparation of this work, the author was supported by the Minerva Research Foundation Member Fund while in residence at the Institute for Advanced Study during the academic year 2025--26.

\bibliographystyle{plain}
\bibliography{biblio}

@book {andre1989g,
	AUTHOR = {Andr\'{e}, Y.},
	TITLE = {{$G$}-functions and geometry},
	SERIES = {Aspects of Mathematics, E13},
	PUBLISHER = {Friedr. Vieweg \& Sohn, Braunschweig},
	YEAR = {1989},
	PAGES = {xii+229},
	ISBN = {3-528-06317-3},
	MRCLASS = {11J82 (11-02 11G10 11J87 12H25)},
	MRNUMBER = {990016},
	MRREVIEWER = {C. L. Stewart},
	DOI = {10.1007/978-3-663-14108-2},
	URL = {https://doi.org/10.1007/978-3-663-14108-2},
}

@article {pilagen,
	AUTHOR = {Pila, J.},
	TITLE = {On a modular {F}ermat equation},
	JOURNAL = {Comment. Math. Helv.},
	FJOURNAL = {Commentarii Mathematici Helvetici. A Journal of the Swiss
	Mathematical Society},
	VOLUME = {92},
	YEAR = {2017},
	NUMBER = {1},
	PAGES = {85--103},
	ISSN = {0010-2571,1420-8946},
	MRCLASS = {11G18 (03C64 11D41 14G35)},
	MRNUMBER = {3615036},
	MRREVIEWER = {Francesc\ Castell\`a},
	DOI = {10.4171/CMH/407},
	URL = {https://doi.org/10.4171/CMH/407},
}

@incollection {fouvrymurty,
	AUTHOR = {Fouvry, \'E. and Murty, M. R.},
	TITLE = {Supersingular primes common to two elliptic curves},
	BOOKTITLE = {Number theory ({P}aris, 1992--1993)},
	SERIES = {London Math. Soc. Lecture Note Ser.},
	VOLUME = {215},
	PAGES = {91--102},
	PUBLISHER = {Cambridge Univ. Press, Cambridge},
	YEAR = {1995},
	ISBN = {0-521-55911-1},
	MRCLASS = {11G05 (11N05)},
	MRNUMBER = {1345175},
	MRREVIEWER = {Eric\ Liverance},
	DOI = {10.1017/CBO9780511661990.007},
	URL = {https://doi.org/10.1017/CBO9780511661990.007},
}

@book {langtrotter,
	AUTHOR = {Lang, S. and Trotter, H.},
	TITLE = {Frobenius distributions in {${\rm GL}\sb{2}$}-extensions},
	SERIES = {Lecture Notes in Mathematics},
	VOLUME = {Vol. 504},
	NOTE = {Distribution of Frobenius automorphisms in ${\rm
	GL}\sb{2}$-extensions of the rational numbers},
	PUBLISHER = {Springer-Verlag, Berlin-New York},
	YEAR = {1976},
	PAGES = {iii+274},
	MRCLASS = {12A50 (10K05)},
	MRNUMBER = {568299},
	MRREVIEWER = {G.\ Frey},
}

@article {masserwuisogellcurves,
	AUTHOR = {Masser, D. W. and W\"ustholz, G.},
	TITLE = {Estimating isogenies on elliptic curves},
	JOURNAL = {Invent. Math.},
	FJOURNAL = {Inventiones Mathematicae},
	VOLUME = {100},
	YEAR = {1990},
	NUMBER = {1},
	PAGES = {1--24},
	ISSN = {0020-9910,1432-1297},
	MRCLASS = {11G05 (11J89 14G25 14K02)},
	MRNUMBER = {1037140},
	MRREVIEWER = {Marc\ Hindry},
	DOI = {10.1007/BF01231178},
	URL = {https://doi.org/10.1007/BF01231178},
}

@misc{papaszpy1,
	title={Zilber-{P}ink in ${Y}(1)^n$: {B}eyond multiplicative degeneration}, 
	author={Papas, G.},
	year={2024},
	eprint={2402.09487},
	archivePrefix={arXiv},
	primaryClass={math.NT},
	url={https://arxiv.org/abs/2402.09487}, 
}

@article{daworrpap,
	title={Some new cases of {Z}ilber-{P}ink in ${Y}(1)^3$},
	author={Daw, C. and Orr, M. and Papas, G.},
	journal={arXiv preprint arXiv:2510.09603},
	year={2025}
}

@article{dawpap,
	title={Lang-{T}rotter phenomena and unlikely intersections},
	author={Daw, C. and Papas, G.},
	journal={available on arXiv},
	year={2026}
}

@incollection {bombg,
	AUTHOR = {Bombieri, E.},
	TITLE = {On {$G$}-functions},
	BOOKTITLE = {Recent progress in analytic number theory, {V}ol. 2 ({D}urham,
	1979)},
	PAGES = {1--67},
	PUBLISHER = {Academic Press, London-New York},
	YEAR = {1981},
	MRCLASS = {10F35 (33A70)},
	MRNUMBER = {637359},
	MRREVIEWER = {F. Beukers},
}

@article{habeggerpila1,
	title={Some unlikely intersections beyond {A}ndr{\'e}--{O}ort},
	author={Habegger, P. and Pila, J.},
	journal={Compositio Mathematica},
	volume={148},
	number={1},
	pages={1--27},
	year={2012},
	publisher={London Mathematical Society}
}

@book {hindrysilverman,
	AUTHOR = {Hindry, M. and Silverman, J.},
	TITLE = {Diophantine geometry},
	SERIES = {Graduate Texts in Mathematics},
	VOLUME = {201},
	NOTE = {An introduction},
	PUBLISHER = {Springer-Verlag, New York},
	YEAR = {2000},
	PAGES = {xiv+558},
	ISBN = {0-387-98975-7; 0-387-98981-1},
	MRCLASS = {11Gxx (11-02 11G10 11G30 11G50 14G25)},
	MRNUMBER = {1745599},
	MRREVIEWER = {Dino J. Lorenzini},
	DOI = {10.1007/978-1-4612-1210-2},
	URL = {https://doi.org/10.1007/978-1-4612-1210-2},
}

@book {pilabook,
	AUTHOR = {Pila, J.},
	TITLE = {Point-counting and the {Z}ilber-{P}ink conjecture},
	SERIES = {Cambridge Tracts in Mathematics},
	VOLUME = {228},
	PUBLISHER = {Cambridge University Press, Cambridge},
	YEAR = {2022},
	PAGES = {x+254},
	ISBN = {978-1-009-17032-1},
	MRCLASS = {11U09 (11-02 11D45 11G35)},
	MRNUMBER = {4420059},
	DOI = {10.1017/9781009170314},
	URL = {https://doi.org/10.1017/9781009170314},
}

@incollection {siegel,
	AUTHOR = {Siegel, C. L.},
	TITLE = {\"{U}ber einige {A}nwendungen diophantischer {A}pproximationen
	[reprint of {A}bhandlungen der {P}reu\ss ischen {A}kademie der
	{W}issenschaften. {P}hysikalisch-mathematische {K}lasse 1929,
	{N}r. 1]},
	BOOKTITLE = {On some applications of {D}iophantine approximations},
	SERIES = {Quad./Monogr.},
	VOLUME = {2},
	PAGES = {81--138},
	PUBLISHER = {Ed. Norm., Pisa},
	YEAR = {2014},
	MRCLASS = {11G30 (01A75 11Jxx)},
	MRNUMBER = {3330350},
}

@article{papaspadicpart1,
	title={On the $v$-adic values of {G}-functions {I}},
	author={Papas, G.},
	journal={},
	year={2025}
}

@article{papaspadicpart2,
	title={On the $v$-adic values of {G}-functions {II}},
	author={Papas, G.},
	journal={},
	year={2025}
}

@article{papaspadicpart3,
	title={On the $v$-adic values of {G}-functions {III}},
	author={Papas, G.},
	journal={},
	year={2025}
}

@article {daworr4,
	AUTHOR = {Daw, C. and Orr, M.},
	TITLE = {Zilber-{P}ink in a product of modular curves assuming
	multiplicative degeneration},
	JOURNAL = {Duke Math. J.},
	FJOURNAL = {Duke Mathematical Journal},
	VOLUME = {174},
	YEAR = {2025},
	NUMBER = {13},
	PAGES = {2877--2926},
	ISSN = {0012-7094,1547-7398},
	MRCLASS = {11G18 (11G50 14G35)},
	MRNUMBER = {4965193},
	DOI = {10.1215/00127094-2025-0011},
	URL = {https://doi.org/10.1215/00127094-2025-0011},
}

@book {silvermanell,
	AUTHOR = {Silverman, J. H.},
	TITLE = {The arithmetic of elliptic curves},
	SERIES = {Graduate Texts in Mathematics},
	VOLUME = {106},
	PUBLISHER = {Springer-Verlag, New York},
	YEAR = {1986},
	PAGES = {xii+400},
	ISBN = {0-387-96203-4},
	MRCLASS = {11G05 (14Gxx 14K07 14K15)},
	MRNUMBER = {817210},
	MRREVIEWER = {Robert S. Rumely},
	DOI = {10.1007/978-1-4757-1920-8},
	URL = {https://doi.org/10.1007/978-1-4757-1920-8},
}

@article {akbaryparks,
	AUTHOR = {Akbary, A. and Parks, J.},
	TITLE = {On the {L}ang-{T}rotter conjecture for two elliptic curves},
	JOURNAL = {Ramanujan J.},
	FJOURNAL = {Ramanujan Journal. An International Journal Devoted to the
	Areas of Mathematics Influenced by Ramanujan},
	VOLUME = {49},
	YEAR = {2019},
	NUMBER = {3},
	PAGES = {585--623},
	ISSN = {1382-4090,1572-9303},
	MRCLASS = {11G05 (11M41)},
	MRNUMBER = {3979693},
	MRREVIEWER = {Joseph\ H.\ Silverman},
	DOI = {10.1007/s11139-018-0050-7},
	URL = {https://doi.org/10.1007/s11139-018-0050-7},
}
\end{document}